\RequirePackage{ifpdf} 
\ifpdf
\documentclass[11pt,pdftex]{article}
\else
\documentclass[11pt,dvips]{article}
\fi

\ifpdf
\fi
\usepackage[bookmarks, 
colorlinks=false, backref,plainpages=false]{hyperref}
\usepackage[english]{babel}
\usepackage[pdftex]{graphicx,xcolor}
\usepackage{caption}
\usepackage{subcaption}
\usepackage{epsfig}
\usepackage{amsmath,amssymb,amsthm}
\usepackage{enumitem}

\renewcommand{\theequation}{\thesection.\arabic{equation}}

\newtheorem{thm}{Theorem}[section]

\newtheorem{lem}[thm]{Lemma}

\newtheorem{rem}[thm]{Remark}

\usepackage{graphics}
\usepackage{epsfig}
\usepackage{amsmath,amssymb,amsthm}

\newcommand{\vertiii}[1]{{\left\vert\kern-0.25ex\left\vert\kern-0.25ex\left\vert #1 
    \right\vert\kern-0.25ex\right\vert\kern-0.25ex\right\vert}}

\newcommand{\bdm}{\begin{displaymath}}
\newcommand{\edm}{\end{displaymath}}
\newcommand{\beq}{\begin{equation}}
\newcommand{\eeq}{\end{equation}}
\newcommand{\beqa}{\begin{eqnarray}}
\newcommand{\eeqa}{\end{eqnarray}}
\newcommand{\beqas}{\begin{eqnarray*}}
\newcommand{\eeqas}{\end{eqnarray*}}

\newcommand{\divv}{{\rm div}}

\newcommand{\gradt}{\nabla\cdot}

\newcommand{\btau}{\mbox{\boldmath$\tau$}}
\newcommand{\bbeta}{\mbox{\boldmath$\beta$}}

\newcommand{\bsigma}{\mbox{\boldmath$\sigma$}}

\newcommand{\bv}{{\bf v}}

\newcommand{\bx}{{\bf x}}

\newcommand{\bE}{{\bf E}}

\newcommand{\bn}{{\bf n}}

\newcommand{\cE}{{\cal E}}
\newcommand{\cT}{{\cal T}}

\newcommand{\cO}{{\cal O}}
\newcommand{\cf}{{\cal F}}

\newcommand{\cu}{{\cal U}}

\newcommand{\Om}{\Omega}

\newcommand{\Ho}{H^1(\Omega)}

\newcommand{\Hdiv}{H(\divv;\,\Omega)}

\title{Adaptive Least-Squares Methods for Convection-Dominated Diffusion-Reaction Problems}

\author{Zhiqiang Cai\thanks{Department of Mathematics, Purdue University,
150 N. University Street, West Lafayette, IN 47907-2067, zcai@math.purdue.edu. This work was supported in part by the National Science Foundation under grants DMS-1217081 and DMS-1522707.
}
\and Binghe Chen\thanks{Wells Fargo Corporate \& Investment Banking, Charlotte, NC 28202-4200, binghe.chen@wellsfargo.com.}
\and Jing Yang\thanks{School of Mathematical Science, Peking University, No.5 Yiheyuan Road Haidian District, Beijing, P.R.China 100871, yangjingmath@pku.edu.cn.}}

\date{}

\begin{document}
\maketitle

\setcounter{page}{1}
\parskip=0pt

\begin{abstract}
This paper studies adaptive least-squares finite element methods for convection-dominated diffusion-reaction problems. 
The least-squares methods are based on the first-order system of the primal and dual variables with various ways of 
imposing outflow boundary conditions. 
The coercivity of the homogeneous least-squares functionals are established,
and the a priori error estimates of the least-squares methods are obtained in a norm that incorporates the streamline derivative. 
All methods have the same convergence rate provided that meshes in the layer regions are fine enough. 
To increase computational accuracy and reduce computational cost, adaptive least-squares methods are  implemented and 
numerical results are presented for some test problems.
\end{abstract}



 
\pagestyle{myheadings}
\thispagestyle{plain}
{ADAPTIVE FOSLS FOR THE CONVECTION-DOMINATED PROBLEMS}

\section{Introduction}
\label{sec:intro}
\setcounter{equation}{0}

Due to the small diffusion coefficient, the solution of the convection-dominated diffusion-reaction problem develops the boundary or interior layers, 
i.e., narrow regions where derivatives of the solution change dramatically. 
It is well known that the conventional numerical methods do not work well on either stability or accuracy for such problems. For example, the standard Galerkin method with continuous linear elements exhibits large spurious oscillation in the boundary layer region. 
Over the decades, many successful numerical methods have been studied and may be roughly grouped into three categories: the mesh-fitted approach, the operator-fitted approach, and the stabilization approach. 
The mesh-fitted approach utilizes the a priori information of the solution including the location and the width of the layer to construct a layer-fitted mesh, e.g., the Shishkin mesh. The operator-fitted approach applies the layer-alike functions as the bases of the approximation space. The stabilization approach adds some stabilization term to the bilinear form. 
For example, the well-known streamline upwind Petrov-Galerkin (SUPG) method \cite{hb:82} adds the original equation tested by the convection term as the stabilization.
For a comprehensive collection of the methods, see \cite{Roos:08} and the references therein.  

Recently, least-squares methods have been intensively studied for fluid flow and elasticity problems (see, e.g., \cite{AziKel:85, bg2, BocGun:95, BocGun:98, brr, clw, cmm1, CaiManMC:97}). 
The least-squares methods minimize certain norms of the residual of the first-order system over appropriate finite element spaces. 
The method always leads to a symmetric positive definite problem, and choices of finite element spaces for the primal and dual variables are not subject to the LBB condition. 
Moreover, one striking feature of the least-squares method is that the value of the least-squares functional at the current
approximation provides an accurate estimates of the true error. 

The application of the least-squares methods to the convection-dominated diffusion-reaction problems is still in its infancy. Reported in \cite{chenfuliqiu} is a new least-squares formulation with inflow boundary conditions weakly imposed and outflow boundary conditions ultra-weakly imposed. This formulation works well on regions away from the boundary layer, even on coarse meshes. However, it does not resolve the boundary layer, which is the primary interest of the problem. This phenomena is also observed in the DG method~\cite{am09}, where the boundary conditions are weakly imposed. These works motivate us to treat outflow boundary conditions in different fashions. In particular, 
we study least-squares method for the convection-dominated diffusion-reaction problem with three different ways to handle the outflow boundary conditions.
The a priori error estimates of finite element approximations based on these formulations are established.
 
The solution of the convection-dominated diffusion-reaction problem usually consists of two parts: the solution of a transport problem ($\epsilon = 0$) and the correction (i.e., the boundary layer). To compute the first part, it is sufficient to use a coarse mesh, while it requires a very fine mesh to resolve the boundary layer. 
Without the a priori information on locations of the layers, this observation motivates the use of adaptive mesh refinement algorithm, which has been 
vastly studied (see, e.g., \cite{ang:95,aabr:13,ber:02,bv:84,dor:96,ver:94}). 
However, many a posteriori error estimators are not suitable for the convection-dominated diffusion-reaction problems, since they depend on the small diffusion parameter. 
To design a robust a posteriori error estimator is non-trivial. Nevertheless, for a least-squares formulation, the a posteriori error estimator is handy, which is simply the value of the least-squares functional at the current approximation. Since the least-squares functional has been computed when solving the algebraic equation, there is no additional cost. Besides, the reliability and the efficiency stem easily from the coercivity and the continuity of the bilinear form, respectively. In this paper, we present numerical results of adaptive mesh refinement algorithms using the least-squares estimator.

The rest of this paper is organized as follows. In section~\ref{sec:dcprob}, we present the convection-dominated diffusion-reaction problem and its first-order linear system. Based on the first-order system, three least-squares formulations are introduced and their coercivity are established in section~\ref{sec:lsform}. Section~\ref{sec:mdn} is a computable counterpart of the previous section, which introduces the computable mesh dependent norms to replace the fractional norms in the least-squares functionals. The main objective of section~\ref{sec:fea} is to establish the a priori error estimates. The adaptive mesh refinement algorithm and the numerical tests are exhibited in section~\ref{sec:AMR} and section~\ref{sec:numtest}, respectively.

\subsection{Notation}
\label{ssec:notation}

We use the standard notation and definitions for the Sobolev
spaces $H^s(\Om)^d$ and $H^s(\partial\Om)^d$ for $s\ge 0$. The
standard associated inner products are denoted by $(\cdot , \,
\cdot)_{s,\Om}$ and $(\cdot , \, \cdot)_{s,\partial\Om}$, and
their respective norms are denoted by $\|\cdot \|_{s,\Om}$ and
$\|\cdot\|_{s,\partial\Om}$. (We suppress the superscript $d$
because the dependence on dimension will be clear by context. We
also omit the subscript $\Om$ from the inner product and norm
designation when there is no risk of confusion.) For $s=0$,
$H^s(\Om)^d$ coincides with $L^2(\Om)^d$. In this case, the inner
product and norm will be denoted by 
$(\cdot,\,\cdot)$ and $\|\cdot\|$, respectively. Finally, we define some spaces
\[
H^{1}_{D}(\Omega) := \{q\in \Ho\,:\,q=0\,\,\mbox{on}\,\,\Gamma_{D}\},
\]
\[
H^{1}_{D^{\pm}}(\Omega) := \{q\in \Ho\,:\,q=0\,\,\mbox{on}\,\,\Gamma_{D^{\pm}}\},
\]
and 
\[
H(\divv;\Om)=\{\bv\in L^2(\Om)^2\, :\,\gradt\bv\in L^2(\Om)\},
\]
which is a Hilbert space under the norm
\[
\|\bv\|_{H(\divv;\,\Om)}=\left(\|\bv\|^2+\|\gradt\bv\|^2
\right)^\frac12.
\]


\section{The convection-diffusion-reaction problem}
\label{sec:dcprob}
\setcounter{equation}{0}

Let $\Om$ be a bounded, open, connected subset in $R^{d}\, (d = 2, 3)$ with a Lipschitz continuous boundary $\partial \Om$. Denote by $\bn = (n_{1},\cdots,n_{d})^{t}$ the outward unit vector normal to the boundary. For a given vector-valued function $\bbeta$, denote by 
\[
\Gamma_{+} = \{ \bx\in\partial\Om\, :\, \bbeta\cdot\bn (\bx) > 0\}\quad \text{and}\quad \Gamma_{-} = \{ \bx\in\partial\Om\, :\, \bbeta\cdot\bn (\bx) < 0\}
\]
the outflow and inflow boundaries, respectively.

Consider the following stationary convection-dominated diffusion-reaction problem:
\begin{eqnarray}\label{eqn:dif}
-\epsilon\,\Delta u + \bbeta\cdot\nabla u+c\, u = f\quad \mbox{in}\,\,\Om,
\end{eqnarray}
where the diffusion coefficient $\epsilon$ is a given small constant, i.e., $0<\epsilon\ll 1$;
and $c$ and $f$ are given scalar-valued functions. For simplicity, we consider 
homogeneous Dirichlet boundary condition:
\begin{eqnarray}\label{bcd:dir:dif}
u|_{\partial \Omega} = 0. 
\end{eqnarray}
For the convection and reaction coefficients, we assume that:
\begin{enumerate}
\item[(1)] $\bbeta \in W^{1}_{\infty}(\Om)^{d}$\, and\, $c \in L^{\infty}(\Om)$ with $\| c\|_{\infty} \leq \gamma$;
\item[(2)] there exists a positive constant $\alpha_{0}$ such that 
\begin{eqnarray}\label{ass:coer}
0 < \alpha_{0} \leq c - \frac{1}{2}\nabla\cdot\bbeta\quad \text{a.e. in}\,\, \Om .
\end{eqnarray}
\end{enumerate}

Introducing the dual variable 
\[
\bsigma = -\epsilon^{1/2}\nabla u,
\] 
(\ref{eqn:dif}) may be rewritten as the following first-order system:
\begin{eqnarray}\label{sys:dif:eq1}
\left\{
\begin{array}{lllc}
\bsigma + \epsilon^{1/2}\nabla u &=& 0 \quad &\mbox{in}\,\,\Om,\\[2mm]
\epsilon^{1/2}\nabla\cdot\bsigma + \bbeta\cdot\nabla u + c\, u & =& f \quad &\mbox{in}\,\,\Om.
\end{array}
\right.
\end{eqnarray}


\section{Least-squares formulations}
\label{sec:lsform}
\setcounter{equation}{0}

In this section, we study three least-squares formulations based on the first-order system 
in (\ref{sys:dif:eq1}) with the inflow boundary conditions imposed strongly. 
These formulations differ in how to handle the outflow boundary conditions. 
More specifically, the outflow boundary conditions are treated strongly for the first one
and weakly for the other two through weighted boundary functionals.

To this end, introduce the following least-squares functionals:
 \begin{eqnarray} \label{fnal:s}
  G_{1}(\btau,\,v;\,f) 
 &=& \| \btau+\epsilon^{1/2}\,\nabla v\|^{2}
         +\|\epsilon^{1/2}\,\nabla\cdot\btau+\bbeta\cdot\nabla v+c\, v-f\|^{2},\\[2mm]
 \label{fnal:w}
 G_{2}(\btau,\,v;\,f)
  &=& G_{1}(\btau,\,v;\,f)+\|\epsilon^{-1/2}\,v\|^{2}_{1/2,\Gamma_{+}}, \\[2mm]
 \label{fnal:normalw}
 \mbox{and}\quad G_{3}(\btau,\,v;\,f) 
  &=&G_{1}(\btau,\,v;\,f)+\|v\|^{2}_{1/2,\Gamma_{+}}.
\end{eqnarray}
Since $\epsilon$ is very small, the outflow boundary conditions are enforced stronger in
$G_2$ than in $G_3$. 
Let
\[
\cu_{1} = H(\divv;\,\Om) \times H^{1}_0(\Om)
\quad\mbox{and}\quad
\cu_{2} = \cu_{3} = H(\divv;\,\Om) \times H^{1}_{\Gamma_{-}}(\Om).
\]
Then the least-squares formulations are to find $(\bsigma,\,u)\in \cu_i$ such that
 \beq\label{lsf}
   G_{i}(\bsigma,\,u;\,f) = \min_{(\btau,\,v)\in\, \cu_i} G_{i}(\btau,\,v;\,f)
 \eeq
for $i=1,\,2,\,3$.

For any $(\btau,\,v)\in \cu_i$, define the following norms: 
 \begin{eqnarray*}
 && M_{1}(\btau,\,v)
 =\|\btau\|^{2}+\| v\|^{2}+\| \epsilon^{1/2}\,\nabla v\|^{2},
 \quad M_{2}(\btau,\,v)= M_1(\btau,\,v)
  +\| \epsilon^{-1/2}\,v\|^{2}_{1/2,\Gamma_{{+}}},\\[2mm]
\mbox{and} &&
 M_{3}(\btau,\,v)= M_1(\btau,\,v)+\|v\|^{2}_{1/2,\Gamma_{{+}}}.
 \end{eqnarray*}
Below we show that the homogeneous least-squares functionals are coercive 
with respect to the corresponding norms. In particular, the coercivity of the functionals
$G_1$ and $G_2$ are independent of the $\epsilon$.

\begin{thm}[Coercivity]\label{thm:2w}\leavevmode
For all $(\btau,\,v)\in \cu_i$ with $i=1,\,2,\,3$, there exist positive constants $C_i$ such that
 \beq\label{coercivity}
 M_{i}(\btau,\,v)
 \leq C_i\, G_{i}(\btau,\,v;\,0),
 \eeq
where $C_1$ and $C_2$ are independent of the $\epsilon$ and $C_3$ is proportional 
to $\epsilon^{-1/2}$.
\end{thm}

\begin{proof}
We provide proofs for $i=2$ and $3$ in detail with an emphasis on how the weight
in $G_2$ leads to the coercivity constant independent of the $\epsilon$. The case of
$i=1$ may be proved in a similar fashion as the case of $i=2$.

For all $(\btau,\,v)\in\cu_{i}$ with $i=1,\,2,\,3$, the triangle inequality gives
\beq\label{3.5a}
 \| \btau\|
 \leq \| \btau+\epsilon^{1/2}\,\nabla v\|
      +\|\epsilon^{1/2}\,\nabla v\|
 \leq G^{1/2}_{1}(\btau,\,v;\,0) + \|\epsilon^{1/2}\,\nabla v\|.
\eeq
Hence, to show the validity of (\ref{coercivity}), it suffices to prove that
 \beq\label{3.6}
 \| v\|^{2}+\|\epsilon^{1/2}\,\nabla v\|^{2} 
 \leq C_i\, G_{i}(\btau,\,v;\,0)
 \quad \forall\,\, (\btau,\,v)\in\cu_{i}.
 \eeq

To this end, let
\beq\label{def:I}
 I = -\Big(\epsilon^{1/2}\,\nabla v,\,\btau\Big)
 +\Big(v,\,(c-\frac{1}{2}\,\nabla\cdot\bbeta)\, v\Big)
 +\dfrac{1}{2}\,\|(\bbeta\cdot\bn)^{1/2}\,v\|^{2}_{0,\Gamma_{{+}}}.
 \eeq
It follows from the definition of the outflow boundary condition and the Cauchy-Schwarz inequality that
\[
\|\epsilon^{1/2}\,\nabla v\|^{2}+\alpha_{0}\,\|v\|^{2}
 \leq (\epsilon^{1/2}\,\nabla v,\,\epsilon^{1/2}\nabla v+\btau)+I 
 \leq \|\epsilon^{1/2}\,\nabla v\| \, G_{1}(\btau,\,v;\, 0)+ I,
\]
which implies
\begin{eqnarray}\label{ineq:ugradu}
 \|\epsilon^{1/2}\,\nabla v\|^{2}+\|v\|^{2} 
 \leq C\,\left(G_{1}(\bsigma,u;0)+ I\right).
\end{eqnarray}
To bound $I$, first note that integration by parts and the boundary conditions imply that
\begin{eqnarray*}\label{eqn:dif:usigw}
 (\epsilon^{1/2}\,\nabla v,\,\btau)
 &=& (v,\,\epsilon^{1/2}\,\btau\cdot\bn)_{\partial\Om}
 -(\epsilon^{1/2}\, v,\, \nabla\cdot\btau)  
 = (v,\,\epsilon^{1/2}\,\btau\cdot\bn)_{\Gamma_{{+}}}
     -(\epsilon^{1/2}\, v,\, \nabla\cdot\btau)\\[2mm] \nonumber
 &=&(v,\,\epsilon^{1/2}\,\btau\cdot\bn)_{\Gamma_{{+}}}
   +(v,\,c\,v)-(v,\,\epsilon^{1/2}\,\nabla\cdot\btau+\bbeta\cdot\nabla v+c\, v)
   +(v\,\bbeta,\,\nabla v)
\end{eqnarray*}
and that
\begin{eqnarray*}\label{eqn:dif:ubuw}
 (\nabla v,\,v\,\bbeta)  
 &=& \dfrac{1}{2}\,\|(\bbeta\cdot\bn)^{1/2}\,v\|^{2}_{0,\Gamma_{{+}}}
     -\dfrac{1}{2}\,\left(v,\,v\,\nabla\cdot\bbeta\right).
\end{eqnarray*}
Combining the above two equalities yields 
\begin{eqnarray}\label{eq:newI}
I = \Big(v,\,\epsilon^{1/2}\,\nabla\cdot\btau+\bbeta\cdot\nabla v+c\, v\Big)
 -(v,\,\epsilon^{1/2}\,\btau\cdot\bn)_{\Gamma_{{+}}}.
\end{eqnarray}

By the trace theorem and the Cauchy-Schwarz inequality, we have
\begin{eqnarray}\nonumber
 \| \btau\cdot\bn\|_{-1/2,\Gamma_{{+}}} 
 &\leq & C\,\Big(\|\btau\|+\|\nabla\cdot\btau\|\Big)\\[2mm]\nonumber
 &\leq&C\,\Big( G^{1/2}_{1}(\btau,\,v;\,0)
   +\|\epsilon^{1/2}\,\nabla v\|
 +\epsilon^{-1/2}\,\|\bbeta\cdot\nabla v\|
 +\epsilon^{-1/2}\, \|c\,v\|\Big)\\[2mm]\label{ineq:tracesig} 
 &\leq& C\,\epsilon^{-1/2}\Big(G^{1/2}_{1}(\btau,\,v;\,0)+\|\nabla v\| +\|v\|\Big).
\end{eqnarray}
Let $\alpha_i= 1$ for $i=2$ or $1/2$ for $i=3$.
Then it follows from (\ref{eq:newI}), the Cauchy-Schwarz inequality, 
the definition of the dual norm, and (\ref{ineq:tracesig}) that for $i=2$ and $3$
\begin{eqnarray}\label{ine:I}
  I 
 &\leq&  \|v\|\, \|\epsilon^{1/2}\,\nabla\cdot\btau+\bbeta\cdot\nabla v+c\, v\|
  +\|\epsilon^{1/2-\alpha_i}\,v\|_{1/2,\Gamma_{{+}}} \,
  \|\epsilon^{\alpha_i}\,\btau\cdot\bn\|_{-1/2,\Gamma_{{+}}}  \\[2mm]\nonumber
 &\leq& C\,\left(\|v\|
 +\|\epsilon^{\alpha_i}\,\btau\cdot\bn\|_{-1/2,\Gamma_{{+}}}\right)
 \,G^{1/2}_{i}(\btau,\,v;\,0)
  \\[2mm]\nonumber
&\leq& C\,G_{i}(\btau,\,v;\,0)
 +C\,\Big(\|\epsilon^{\alpha_i-1/2}\,\nabla v\| 
 +\|v\|\Big)G^{1/2}_{i}(\btau,\,v;\,0),
\end{eqnarray}
which, together with (\ref{ineq:ugradu}), implies
\begin{eqnarray}
\|\epsilon^{1/2}\,\nabla v\|^{2}+\alpha_{0}\,\|v\|^{2} 
 \leq C_i\,G_{i}(\btau,\,v;\,0)
\end{eqnarray}
with $C_2$ independent of $\epsilon$ and $C_3$ proportional to $\epsilon^{-1/2}$.
This completes the proof of (\ref{3.6}) and, hence, (\ref{coercivity}) for $i = 2$ and $3$.

The validity of (\ref{coercivity}) for $i = 1$ may be established in a similar fashion 
by noticing that 
the boundary term of $I$ in (\ref{def:I}) vanishes due to the boundary conditions.
This completes the proof of the theorem.
\end{proof}

\section{Mesh-dependent least-squares functionals}
\label{sec:mdn}
\setcounter{equation}{0}

For computational feasibility, in this section, we replace the $\frac{1}{2}$-norm in the 
least-squares functionals defined in (\ref{fnal:w}) and (\ref{fnal:normalw})  by mesh-dependent
$L^2$-norms. 
For the simplicity of presentation, assume that the domain $\Om$ is a convex polygon in the two dimensional plane. (The extension to the higher dimension is straightforward.) 
Let $\cT_h =\{K\}$ be a triangulation of $\Om$ with triangular elements $K$ of diameter less than or equal to $h$. Assume that the triangulation $\cT_{h}$ is regular and quasi-uniform (see \cite{Cia:78}). 

Denote by $\cE_h$ the set of all edges of the triangulation $\cT_h$. 
The least-squares functionals $G_2$ and $G_3$ defined in (\ref{fnal:w}) and (\ref{fnal:normalw})
are modified by the following computable least-squares functionals:
\begin{eqnarray} \label{func:dispw}
 G^{h}_{2}(\btau,\,v;\,f)
  &=& G_{1}(\btau,\,v;\,f)
 +\sum_{e\in\cE_h\cap\Gamma_{{+}}}h^{-1}_{e}\|\epsilon^{-1/2}\,v\|^{2}_{0,e} \\[2mm]
 \label{func:disw}
 \mbox{and}\quad G^{h}_{3}(\btau,\,v;\,f) 
  &=&G_{1}(\btau,\,v;\,f)+\sum_{e\in\cE_h\cap\Gamma_{{+}}}h^{-1}_{e}\|v\|^{2}_{0,e},
\end{eqnarray}
where $h_e$ denotes the diameter of the edge $e$.

For any triangle $K\in \cT_h$, let $P_k(K)$ be the space of
polynomials of degree less than or equal to $k$ on $K$ and denote the local
Raviart--Thomas space of index $k$ on $K$ by
\[
RT_k(K)=P_k(K)^2 +\left(\begin{array}{c}
x_1\\x_2\end{array}\right)P_k(K).
\]
Then the standard $\Hdiv$ conforming Raviart--Thomas space of index
$k$ \cite{rt} and the standard (conforming) continuous piecewise
polynomials of degree $k+1$ are defined, respectively, by
\begin{eqnarray}\label{S_h}
\qquad\quad \Sigma^k_h&\!\!=\!\!&\{\btau\in H(\divv;\Om)\,:\,
\btau|_K\in RT_k(K),\,\,\forall\,K\in\cT_h\},\\[2mm]
\label{V_h} \qquad\quad
V^{k+1}_h&\!\!=\!\!&\{v\in H^1(\Om)\,:\, v\in
P_{k+1}(K),\,\,\forall\,K\in\cT_h\}.
\end{eqnarray}
These spaces have the following approximation properties: let $k\ge 0$ be an integer,
and let $l\in (0,\,k+1]$:
\begin{equation}\label{app1}
\inf_{\btau\in\,\Sigma^k_h}\|\bsigma -\btau\|_{\Hdiv}
\leq C\,h^l\left(\|\bsigma\|_{l}
+\|\gradt\bsigma\|_l\right)
\end{equation}
for $\bsigma\in H^{l}(\Om)^{2} \cap H(\divv;\Om)$ with
$\gradt\bsigma \in H^{l}(\Om)$ and
\begin{equation}\label{app2}
\inf_{v\in V^{k+1}_h}\|u -v\|_1
\leq C\,h^l\,\|u\|_{l+1}
\end{equation}
for $u\in H^{l+1}(\Om)$. 
In the subsequent sections, based on the smoothness of $\bsigma$ and $u$, we will choose $k+1$ to be the smallest integer greater than or equal to $l$. Since the triangulation $\cT_{h}$ is regular, the following inverse inequalities hold for all $K\in\cT_h$:
\begin{eqnarray}\label{ieqn:inv:2s}
 \| \btau \|_{1,K} 
 &\leq& C\, h^{-1}_{K}\,\| \btau \|_{K},\quad\forall\,\,\btau\in RT_{k}(K)\\[2mm]
 \|v\|_{1,K} 
 &\leq& C\, h^{-1}_{K}\,\|v \|_{K},\quad\forall\,\, v\in P_{k}(K)
\end{eqnarray}
with positive constant $C$ independent of $h_{K}$.

Denote by $\cu_{i}^h$ the finite dimensional subspaces of $\cu_{i}$:
\begin{eqnarray}\label{def:disspace}
\cu_{i}^h = \big(\Sigma^{k}_{h}\times V^{k+1}_{h}\big) \cap \cu_{i}.
\end{eqnarray}
For any $(\btau,\,v)\in \cu_i^h$, define the following norms: 
 \begin{eqnarray*}
  M^h_{2}(\btau,\,v)
  &=& M_1(\btau,\,v)
  +\sum_{e\in\cE_h\cap\Gamma_{{+}}} h^{-1}_{e}\|\epsilon^{-1/2}\,v\|^{2}_{0,e}
 \\[2mm]
\mbox{and}\quad
 M^h_{3}(\btau,\,v)
 &=& M_1(\btau,\,v) 
 + \sum_{e\in\cE_h\cap\Gamma_{{+}}} h^{-1}_{e}\|v\|^{2}_{0,e}.
 \end{eqnarray*}
Below we establish the discrete version of Theorem~\ref{thm:2w}, i.e., the coercivity of the discrete functionals (\ref{func:dispw}) and (\ref{func:disw}) with respect to the norms defined
above.
For the consistence of notation, we also let $G^{h}_{1}=G_{1}$ and $M^{h}_{1}=M_{1}$.

\begin{thm}\label{thm:compnorm}
For all $(\btau,\,v)\in \cu^h_{i}$ with 
$i = 2$ and $3$, there exist positive constants $C_i$ independent of $\epsilon$ such that
\begin{eqnarray}\label{ineq:discoer}
 M^{h}_{i}(\btau,\,v)
  \leq C_i\,G^{h}_{i}(\btau,\,v;\,0).
\end{eqnarray}
\end{thm}

\begin{proof}
Similar to the argument in the proof of Theorem~\ref{thm:2w}, 
in order to establish (\ref{ineq:discoer}), it suffices to show that
\begin{eqnarray}\label{ineq:ugradu:dis}
\|\epsilon^{1/2}\,\nabla v\|^{2} + \| v\|^{2} \leq C\,G^{h}_{i}(\btau,\,v;\,0)
\end{eqnarray}
for all $(\btau,\,v)\in \cu^h_{i}$. Moreover, we have
\begin{eqnarray}\label{ineqn:ugraduI}
\|\epsilon^{1/2}\,\nabla v\|^{2} + \| v\|^{2} \leq C\,\left(G^{h}_{i}(\btau,\,v;\,0)+I\right)
\end{eqnarray}
with $I$ defined in (\ref{def:I}). 

For any $e\in \cE_h\cap \Gamma_+$, let $e$ be an edge of element $K\in\cT_h$.
It follows from the trace theorem and the inverse inequality in (\ref{ieqn:inv:2s}) that
  \[
 h_e\,\|\btau\cdot\bn\|^2_{0,e}
 \leq C\,h_e\, \|\btau\|^2_{0,e}
 \leq C\, h_e\,\|\btau\|_{0,K}\|\btau\|_{1,K}
 \leq C\, \|\btau\|_{0,K}^2,
  \]
which, together with (\ref{3.5a}), implies
 \beq\label{4.13}
 \left(\sum_{e\in \cE_h\cap \Gamma_+} h_e\,\|\btau\cdot\bn\|^2_{0,e}\right)^{1/2}
 \leq C\, \|\btau\|
  \leq C\, \left( G^{1/2}_1(\btau,\,v;\,0)  +\|\epsilon^{1/2}\,\nabla v\|\right).       
 \eeq

Let $\alpha_i=1$ for $i=2$ or $1/2$ for $i=3$.
It follows from (\ref{eq:newI}), the Cauchy-Schwarz inequality, and (\ref{4.13}) that
\begin{eqnarray*}
 I 
 &=& \Big(v,\,\epsilon^{1/2}\,\nabla\cdot\btau+\bbeta\cdot\nabla v+c\,v\Big)
 -(v,\,\epsilon^{1/2}\,\btau\cdot\bn)_{\Gamma_{{+}}}\\[2mm]\nonumber
  &\leq& C\, \left(\|v\|
  +\epsilon^{\alpha_i}\,
 \Big(\sum_{e\in\cE_h\cap\Gamma_{{+}}} 
   h_{e}\,\|\btau\cdot\bn\|^{2}_{0,e}\Big)^{1/2}\right)\,
  G^{h}_{i}(\btau,\,v;\,0)^{1/2}\\[2mm]\nonumber
 &\leq& C\, G^{h}_{i}(\btau,\,v;\,0)
 +C\, \left(\|v\|+\|\epsilon^{1/2}\nabla v\|\right) 
 \,G^{h}_{i}(\btau,\,v;\,0)^{1/2}
\end{eqnarray*}
which, together with (\ref{ineqn:ugraduI}), implies the validity of (\ref{ineq:ugradu:dis})
and, hence, (\ref{ineq:discoer}).
This completes the proof of the theorem.
\end{proof}

\begin{rem}
Note that the coercivity constant $C_3$ in the discrete version is no longer depending 
on $\epsilon$, that is better than the continuous version (see Theorem~\ref{thm:2w}).
\end{rem}

\section{Finite element approximations}
\label{sec:fea}
\setcounter{equation}{0}

The least-squares problems are to find $(\bsigma,\,u)\in \cu_{i}$ ($i = 1,\,2,\,3$) such that
\begin{eqnarray}
 G^{h}_{i}(\bsigma,\,u;\,f) = \min_{(\btau,\,v)\in\, \cu_{i}} G^{h}_{i}(\btau,\,v;\,f).
\end{eqnarray}
The corresponding variational problems are to find $(\bsigma,\,u) \in \cu_{i}$ such that
\begin{eqnarray}\label{eqn:vf}
 a_{i}(\bsigma,\,u;\,\btau,\,v) = \cf_{i}(\btau,\,v), \quad \forall\, (\btau,\,v)\in \cu_{i},
\end{eqnarray}
where the bilinear forms $a_{i}(\cdot\, ;\, \cdot)$ are symmetric and given by
\begin{eqnarray*}
a_{1}(\bsigma,\,u;\,\btau,\,v) &=& (\bsigma+\epsilon^{1/2}\,\nabla u,\,\btau+\epsilon^{1/2}\,\nabla v) \\[2mm] \nonumber
& & +(\epsilon^{1/2}\,\nabla\cdot\bsigma+\bbeta\cdot\nabla u+c\,u,\,\epsilon^{1/2}\,\nabla\cdot\btau+\bbeta\cdot\nabla v+c\,v),\\[2mm]
a_{2}(\bsigma,\,u;\,\btau,\,v) &=& a_{1}(\bsigma,\,u;\,\btau,\,v) +\sum_{e\,\in\cE_{h}\cap\Gamma_{{+}}}h^{-1}_{e}\,\epsilon^{-1}\,(u,\,v)_{0,e},\\[2mm]
a_{3}(\bsigma,\,u;\,\btau,\,v) &=& a_{1}(\bsigma,\,u;\,\btau,\,v) +\sum_{e\,\in\cE_{h}\cap\Gamma_{{+}}}h^{-1}_{e}\,(u,\,v)_{0,e},
\end{eqnarray*}
and the linear forms $\cf_{i}(\cdot)$ are given by
\begin{eqnarray*}\label{def:linfunc}
\cf_{i}(\btau,\,v) = (f,\,\epsilon^{1/2}\,\nabla\cdot\btau+\bbeta\cdot\nabla v+c\,v)\quad\,\mbox{for}\,\, i = 1,\,2,\,3.
\end{eqnarray*}

The least-squares finite element approximations to the variational problems in (\ref{eqn:vf}) are to find $(\bsigma^i_{h},\,u^i_{h}) \in \cu^{h}_{i}$ such that
\begin{eqnarray}\label{eqn:varprom}
a_{i}(\bsigma^i_{h},\,u^i_{h};\,\btau,\,v) = \cf_{i}(\btau,\,v), \quad \forall\,\, (\btau,\,v)\in \cu^{h}_{i},
\end{eqnarray}
for $i = 1,\,2,\,3$.
Taking the difference between (\ref{eqn:vf}) and (\ref{eqn:varprom}) implies the following orthogonality:
\begin{eqnarray}\label{eqn:ortho}
a_{i}(\bsigma - \bsigma^i_{h},\,u - u^i_{h};\,\btau,\,v) = 0,\quad \forall \,\, (\btau,\,v)\in \cu^{h}_{i}.
\end{eqnarray}

In the rest of this section, we consider a stronger norm which incorporates the norm of the streamline derivative:
\[
\vertiii{(\btau,\,v)}^{2}_{i} = M^{h}_{i}\,(\btau,\,v)+\sum_{K\in\cT_{h}}\delta_{K}\,\|\bbeta\cdot\nabla v\|^{2}_{K},
\]
where $\delta_{K}$ is a positive constant to be determined. 
In the following lemma, we show that $G^{h}_{i}(\bsigma,\,u;\,0)$ are also elliptic with respect to these norms 
if the $\delta_{K}$ is appropriately chosen.

\begin{lem}\label{lmm:disnorm:2w}
For all $K\in\cT_{h}$, assume that $0<\delta_{K} \leq \min\{h^{2}_{K}/\epsilon,\,C\}$, then there exist positive constants $C_{i}$ independent of $\epsilon$ such that
\begin{eqnarray}\nonumber
\vertiii{ (\btau,\,v)}^{2}_{i} \leq C_{i}\, G^{h}_{i}\,(\btau,\,v;\,0),\quad \forall\, (\btau,\,v)\in \cu^{h}_{i}, \quad i = 1,\,2,\,3.
\end{eqnarray}
\end{lem}
\begin{proof}
By Theorems~\ref{thm:2w} and \ref{thm:compnorm}, to prove the validity of the lemma, it suffices to show that 
\begin{eqnarray}\label{ineq:dirnorm}
\sum_{K\in\cT_{h}}\delta_{K}\,\|\bbeta\cdot\nabla v\|^{2}_{K} \leq C_{i}\, G^{h}_{i}\,(\btau,\,v;\,0).
\end{eqnarray}
To this end, note the facts that 
 \[
   \delta_K\leq C
  \quad\mbox{and}\quad
  \dfrac{\delta_{K}\,\epsilon }{h^{2}_{K}} 
  \leq \min\left\{1,\, \dfrac{C\,\epsilon}{h^{2}_{K}}\right\} \leq C.
\]
Now it follows from the Cauchy-Schwarz inequality and the inverse inequality in (\ref{ieqn:inv:2s}) that
 \begin{eqnarray*}\nonumber
 \sum_{K\in\cT_{h}}\delta_{K}\,\|\bbeta\cdot\nabla v\|^{2}_{K}
 & \leq & C\, \sum_{K\in\cT_{h}} \delta_{K}\,\left(G^{h}_{1,K}\,(\btau,\,v;\,0) + \|\epsilon^{1/2}\,\nabla\cdot\btau\|^{2}_{K}+\|c\,v\|^{2}_{K}\right)\\[2mm]\nonumber
 &\leq& C\, \sum_{K\in\cT_{h}} \left(G^{h}_{1, K}\,(\btau,\,v;\,0) + \dfrac{\delta_K\,\epsilon}{h^{2}_{K}}\,\|\btau\|^{2}_{K}+\|v\|^{2}_{K}\right)\\[2mm]\nonumber
&\leq& C\,\left(G^{h}_{1}\,(\btau,\,v;\,0)+\|\btau\|^2 + \|v\|^2 \right)    
\leq C\,G^{h}_{i}\,(\btau,\,v;\,0),
\end{eqnarray*}
which establishes (\ref{ineq:dirnorm}) and hence completes the proof of the lemma.
\end{proof}

To choose $\delta_K$ properly, first define 
the local mesh P\'eclet number by
\[
{P\!e}_{K} = \frac{\|\bbeta\|_{0,\infty,K}\, h_{K}}{2\,\epsilon},
\]
then partition the triangulation $\cT_{h}$ into two subsets:
\begin{eqnarray}\label{ntion:condodiffco}
  \cT^{c}_{h} = \{K\in\cT_{h}\,:\, {P\!e}_{K} > 1\}
  \quad\text{and}\quad
  \cT^{d}_{h} = \{K\in\cT_{h}\,:\,{P\!e}_{K} \leq 1\}.
\end{eqnarray}
The elements in $\cT^{c}_{h}$ are referred to the convection-dominated elements, while the elements in $\cT^{d}_{h}$ the diffusion-dominated elements.
Now, the $\delta_{K}$ is chosen to be
\begin{eqnarray}\label{def:deltaK}
\delta_{K} = \left\{
\begin{array}{ll}
\dfrac{2\,h_{K}}{\|\bbeta\|_{0,\infty,K}}, & \text{if}\,\, K \in \cT^{c}_{h},\\[8mm]
\dfrac{h^{2}_{K}}{\epsilon}, & \text{if}\,\,K \in \cT^{d}_{h}.
\end{array}
\right.
\end{eqnarray}

\begin{rem}
The $\delta_{K}$ defined in {\em (\ref{def:deltaK})} satisfies the assumption in 
{\em Lemma~\ref{lmm:disnorm:2w}}, i.e., 
 \begin{equation}\label{Delta}
 \delta_K \le \min\{h^{2}_{K}/\epsilon,\,C\}.
 \end{equation}
\end{rem}

\begin{proof}
Since $\|\bbeta\|_{0,\infty,K}$ is large comparing to $h_K$, we have 
 \begin{equation}\label{C}
  \frac{2\,h_{K}}{\|\bbeta\|_{0,\infty,K}} 
   \leq C.
   \end{equation}
For any $K\in\cT^{c}_{h}$, the fact that ${P\!e}_{K} > 1$ implies 
 \[
  \frac{2\,h_{K}}{\|\bbeta\|_{0,\infty,K}} 
   < \dfrac{h^2_K}{\epsilon},
   \]
   which, together with (\ref{C}), yields (\ref{Delta}).
For any $K\in\cT^{d}_{h}$, (\ref{Delta}) is again a consequence of the definition of $\delta_K$ in 
(\ref{def:deltaK}), the fact that ${P\!e}_{K} \leq 1$, and (\ref{C}).
\end{proof}

Denote by $\cT^{\partial}_{h}$ the set of elements that intersect the outflow boundary nontrivially, i.e.,
\[
\cT^{\partial}_{h} = \{K\in\cT_{h}\,:\, meas(\bar{K}\cap \Gamma^{+}) > 0\}.
\]
In this paper, we assume that
\begin{eqnarray}\label{ass:elebdry}
\cT^{\partial}_{h} \subset \cT^{d}_{h}.
\end{eqnarray}
For any $K\in \cT^{d}_{h}$, the fact that ${P\!e}_{K} \leq 1$ implies
 \[
  h_K < \dfrac{2\, \epsilon}{\|\bbeta\|_{0,\infty,K}}.
  \]
Hence, assumption (\ref{ass:elebdry}) means that the mesh size in the boundary layer region 
is comparable to the perturbation parameter $\epsilon$. 

\begin{thm}\label{thm:apierr:w}
Let $(\bsigma,\,u)$ be the solution of {\em (\ref{eqn:vf})}.
Assume that $(\bsigma,\,u) \in H^{l}(\Om)^{2}\times H^{l+1}(\Om)$ and that $\nabla\cdot\bsigma\in H^{l}(\Om)$.
Let $(\bsigma^i_{h},\,u^i_{h})$,  $i=1,\,2,\,3$, be the solution of {\em (\ref{eqn:varprom})} with $k=l$. 
Under the assumption in {\em (\ref{ass:elebdry})}, we have the following a priori error estimation:
 \begin{eqnarray} \nonumber
 && C_{i}\,\vertiii{(\bsigma-\bsigma^i_h,\,u-u^i_h)}^{2}_{i} \\ [2mm] \nonumber
 &\leq& \sum_{K\in\cT^{c}_{h}}  h^{2l-1}_{K} \, \left(\epsilon \,\|\nabla\cdot\bsigma\|^{2}_{l,K}+h_{K}\,
    \|\bsigma\|^{2}_{l,K}+ \|u\|^{2}_{l+1,K}\right)\\[2mm] \label{thm:5.3}
 &&+\!\!\sum_{K\in\cT^{d}_{h}} h^{2l-1}_{K}\, \left( \dfrac{ \epsilon^{2}} {h_{K}}\,
   \|\nabla\cdot\bsigma\|^{2}_{l,K}
    +h_{K}\,\|\bsigma\|^{2}_{l,K}+   \dfrac{ \epsilon} {h_{K}}\, \|u\|^{2}_{l+1,K}\right),
 \end{eqnarray}
where constants $C_{i}>0$ are independent of $\epsilon$.
\end{thm}
\begin{proof}
We provide proof of (\ref{thm:5.3}) only for $i=2$ and $3$ since (\ref{thm:5.3}) may be obtained in a
similar fashion.

To this end, let $\bsigma_{I}$ and $u_{I}$ be the interpolants of $\bsigma$ and $u$, respectively, such that
the approximation properties in (\ref{app1}) and (\ref{app2}) hold and that
 \begin{equation}\label{commut}
 \left(\nabla\cdot (\bsigma -\bsigma_I),\, v\right) = 0,
 \quad\forall\,\, v\in D^h_{k},
 \end{equation}
where $D^h_k =\{ v\in L^2(\Omega)\, : \, v|_K\in P_k(K)\,\,\forall\, K\in \cT_h\}$ is the space of discontinuous 
piecewise polynomials of degree less than or equal to $k\ge 0$.
Let
 \begin{eqnarray*}
 \bE_I=\bsigma-\bsigma_I,\quad
 \bE^i_h=\bsigma_I-\bsigma^i_h, \quad
 e_I=u-u_I,\quad\mbox{and}\quad
 e^i_h=u_I-u^i_h.
 \end{eqnarray*}
Since $\bE^i= \bsigma-\bsigma^i_h = \bE_I+\bE^i_h$ and $e^i = u-u^i_h = e_I+e^i_h$,
 the triangle inequality gives
 \begin{eqnarray}\label{ineq:errtri}
 \vertiii{(\bE^i,\, e^i)}_{i} \leq \vertiii{(\bE_{I}, \,e_{I})}_{i} + \vertiii{(\bE^i_{h}, \,e^i_{h})}_{i}.
 \end{eqnarray}
 Let $\alpha_i=-1$ or $0$ for $i=2, \, 3$. By approximation property (\ref{app2})
 and assumption (\ref{ass:elebdry}), we have
 \[
 \sum_{e\,\in\cE_{h}\cap\Gamma_{{+}}}
  h^{-1}_{e}\,\epsilon^{\alpha_i}\,\|e_{I}\|^{2}_{0,e} 
 \leq C\,\sum_{K\in\cT^{\partial}_{h}}
  h^{2l}_{K}\,\epsilon^{\alpha_i}\,\|u\|^{2}_{l+1,K}
 \leq C\,\sum_{K\in\cT^{\partial}_{h}}h^{2l+\alpha_{i}}_{K}\,\|u\|^{2}_{l+1,K}. 
 \]
Now, it follows from (\ref{app1}), (\ref{app2}), the trace theorem, 
and the fact $\delta_{K} \le C$ that
 \begin{eqnarray}\nonumber
 & &\vertiii{(\bE_{I},\,e_{I})}_{i}^2\\[2mm]\nonumber
 &\leq&C\,\left( \|\bE_{I}\|^2+\|e_{I}\|^2
 +\|\epsilon^{1/2}\,\nabla e_{I}\|^2
 +\sum_{e\in\Gamma_{{+}}}h^{-1}_{e}\,\epsilon^{\alpha_{i}}\,\|e_{I}\|^2_{e}
 +\sum_{K\in\cT_{h}}\|\bbeta\cdot\nabla e_{I}\|^{2}_{K} \right)\\[2mm]\label{ineqn:intererr:23w}
 &\leq& C\,\left(\sum_{K\in\cT_{h}} h^{2l}_K\,\|\bsigma\|^2_{l,K}
  +\sum_{K\in\cT_{h}}h^{2l}_K\,\|u\|^2_{l+1,K}+\sum_{K\in\cT^{\partial}_{h}}h^{2l+\alpha_{i}}_{K}\,\|u\|^{2}_{l+1,K}\right).
\end{eqnarray}

To bound the second term of the right-hand side in (\ref{ineq:errtri}), by Lemma~\ref{lmm:disnorm:2w}
and orthogonality (\ref{eqn:ortho}), we have
\begin{eqnarray}\label{ieqn:diserr:w}
 C_{i}\,\vertiii{(\bE^{i}_{h},\, e^{i}_{h})}^{2}_{i}
 \leq a_{i}(\bE^{i}_{h},\,e^{i}_{h};\,\bE^{i}_{h},\,e^{i}_{h}) = a_{i}(\bE^{i}_{h},\,e^{i}_{h};\,-\bE_{I},\,-e_{I})
 \equiv I^{i}_{1}+I^{i}_{2}+I^{i}_{3}+I^{i}_{4}, \,\,\,\,
\end{eqnarray}
where
 \begin{eqnarray}\nonumber
  I^{i}_{1} 
 &=& (c\,e^{i}_{h},\,-\epsilon^{1/2}\,\nabla\cdot\bE_{I}-\bbeta\cdot\nabla e_{I}-c\,e_{I})
   +(\bE^{i}_{h}+\epsilon^{1/2}\,\nabla e^{i}_{h},\, -\bE_{I}-\epsilon^{1/2}\,\nabla e_{I}),\\[2mm]\nonumber
 I^{i}_{2} 
 &=& (\epsilon^{1/2}\,\nabla\cdot\bE^{i}_{h},\,-\epsilon^{1/2}\,\nabla\cdot\bE_{I}-\bbeta\cdot\nabla e_{I}-c\,e_{I}),\\[2mm]\nonumber
 I^{i}_{3} 
 &=& (\bbeta\cdot\nabla e^{i}_{h},\,-\epsilon^{1/2}\,\nabla\cdot\bE_{I}-\bbeta\cdot\nabla e_{I}-c\,e_{I}),\\[2mm]\nonumber
 \mbox{and}\quad 
 I^{i}_{4} 
 &=& \sum_{e\,\in\cE_{h}\cap\Gamma_{{+}}}h^{-1}_{e}\,\epsilon^{\alpha_{i}}\,(e^{i}_{h},\,-e_{I})_{0,e}.
\end{eqnarray}
It follows from the triangle and Cauchy-Schwarz inequalities,  (\ref{app1}), 
and (\ref{app2}) that
\begin{eqnarray}\nonumber\label{ineqn:I1:w}
  && I^{i}_{1}  \\[2mm]\nonumber
 &\leq& \!\!\!
  C\, \|e^{i}_{h}\|
 \left(\|\epsilon^{1/2}\nabla\cdot\bE_{I}\|+\|\nabla e_{I}\|+\|e_{I}\|\right)
 +C\,  \left(\|\bE^{i}_{h}\|+\|\epsilon^{1/2}\nabla e^{i}_{h}\|\right)
   \left(\|\bE_{I}\|+\|\epsilon^{1/2}\nabla e_{I})\|\right)\\[2mm]
 &\leq& \!\!\!\!\!
  C \left(\|e^{i}_{h}\|+\|\bE^{i}_{h}\|+\|\epsilon^{1/2}\nabla e^{i}_{h}\|\right)
 \left(\sum_{K\in\cT_{h}}h^{2l}_{K}\Big(\epsilon\|\nabla\cdot\bsigma\|^{2}_{l,K}
   +\|\bsigma\|^{2}_{l,K}+\|u\|^{2}_{l+1,K} \Big)  \!\!\right)^{1/2}. \quad\,\,
\end{eqnarray}
By (\ref{commut}), the Cauchy-Schwarz and triangle inequalities, and the
inverse inequality in (\ref{ieqn:inv:2s}), we have
\begin{eqnarray}\nonumber
  && I^{i}_{2} 
  = - (\epsilon^{1/2}\,\nabla\cdot\bE^{i}_{h},\,   
   \bbeta\cdot\nabla e_{I} + c\,e_{I}),\\[2mm] \label{ineqn:I2:w}
  &\leq & \!\!\!\! C \sum_{K\in\cT_{h}}\frac{\epsilon^{1/2}}{h_{K}}
        \|\bE^{i}_{h}\|_{K}\Big(\|\nabla e_{I}\|_{K}+\|e_{I}\|_{K}\Big)
  \leq C\, \|\bE^{i}_{h}\|\left(\sum_{K\in\cT_{h}}\epsilon\, h_{K}^{2l-2}\|u\|^{2}_{l+1,K}\right)^{1/2}. \quad \qquad 
\end{eqnarray}
By the Cauchy-Schwarz and the triangle inequalities, $I_{3}$ is bounded by
\begin{eqnarray} \nonumber
  &&  I^{i}_{3} %
 \leq C\sum_{K\in\cT_{h}} \|\bbeta\cdot\nabla e^{i}_{h}\|_{K}
 \left(\epsilon^{1/2}\,\|\nabla\cdot\bE_{I}\|_{K}
  +\|\nabla e_{I}\|_{K}+\|e_{I}\|_{K}\right)\\[2mm]\nonumber
 &\leq&C\sum_{K\in\cT_{h}} \|\bbeta\cdot\nabla e^{i}_{h}\|_{K}
 \left(\epsilon^{1/2}\,h^{l}_{K}\,\|\nabla\cdot\bsigma\|_{l,K}
 +h^{l}_{K}\,\|u\|_{l+1,K}\right)\\[2mm] \label{ineqn:I3:w}
 &\leq& \!\!\!  C\left(\sum_{K\in\cT_{h}} 
  \delta_{K}\|\bbeta\cdot\nabla e^{i}_{h}\|^{2}_{K}\right)^{1/2} \!\!
 \left(\sum_{K\in\cT_{h}}\delta^{-1}_{K}\,
 \left(\epsilon\, h^{2l}_{K}\,\|\nabla\cdot\bsigma\|^{2}_{l,K}
 +h^{2l}_{K}\,\|u\|^{2}_{l+1,K}\right)  \!\! \right)^{1/2}. \quad \quad
\end{eqnarray}
For $I^{i}_{4}$, it follows from the Cauchy-Schwarz inequality and the trace theorem that
\begin{eqnarray}\nonumber
 I^{i}_{4} %
 &\leq& C\,\left(\sum_{e\,\in\cE_{h}\cap\Gamma_{{+}}} h^{-1}_{e}\,
 \epsilon^{\alpha_{i}}\,\|e^{i}_{h}\|^{2}_{0,e}\right)^{1/2}
 \left(\sum_{e\,\in\cE_{h}\cap\Gamma_{{+}}}h^{-1}_{e}\,
 \epsilon^{\alpha_{i}}\,\|e_{I}\|^{2}_{0,e}\right)^{1/2}\\[2mm] \label{ineqn:I4:w}
 &\leq& C\,\left(\sum_{e\,\in\cE_{h}\cap\Gamma_{{+}}}h^{-1}_{e}\,
 \epsilon^{\alpha_{i}}\,\|e^{i}_{h}\|^{2}_{0,e}\right)^{1/2}
  \left(\sum_{K\in\cT^{\partial}_{h}}h_{K}^{2l+\alpha_{i}}\,\|u\|^{2}_{l+1,K}\right)^{1/2}.
\end{eqnarray}
Combining (\ref{ieqn:diserr:w}), (\ref{ineqn:I1:w}), (\ref{ineqn:I2:w}), (\ref{ineqn:I3:w}), (\ref{ineqn:I4:w}), and (\ref{Delta}),
 we have
\begin{eqnarray*}
 & &C_{i}\,\vertiii{(\bE^{i}_{h},\, e^{i}_{h})}^{2}_{i}\\[2mm]
 &\leq & \sum_{K\in\cT_{h}}h^{2l}_{K}\|\bsigma\|^{2}_{l,K} 
 + \sum_{K\in\cT_{h}} \left(1+\delta^{-1}_{K}\right)\,\epsilon\, h^{2l}_{K}\,\|\nabla\cdot\bsigma\|^{2}_{l,K} + \sum_{K\in\cT^{\partial}_{h}}h_{K}^{2l+\alpha_{i}}\,\|u\|^{2}_{l+1,K}\\[2mm]
 & &+ \sum_{K\in\cT_{h}} \left(1 + \epsilon\, h^{-2}_{K}+\delta_K^{-1}\right) h_{K}^{2l}\|u\|^{2}_{l+1,K}\\[2mm]
 &\leq& \sum_{K\in\cT_{h}}
 \left(\frac{\epsilon\, h^{2l}_{K}}{\delta_{K}}\,\|\nabla\cdot\bsigma\|^{2}_{l,K}+h^{2l}_{K}\,\|\bsigma\|^{2}_{l,K}
 +\frac{h^{2l}_{K}}{\delta_{K}}\,\|u\|^{2}_{l+1,K}\right)\\[2mm]
 & &+\sum_{K\in\cT^{\partial}_{h}}h_{K}^{2l+\alpha_{i}}\,\|u\|^{2}_{l+1,K},
\end{eqnarray*}
which, together with the definition of $\delta_{K}$ in (\ref{def:deltaK}), implies
\begin{eqnarray*}\label{ineq:uhuI2}
   C_{i}\,\vertiii{(\bE^{i}_{h},\, e^{i}_{h})}^{2}_{i} 
  &\leq& \sum_{K\in\cT^{c}_{h}}  h^{2l-1}_{K}
    \Big(\epsilon\,\|\nabla\cdot\bsigma\|^{2}_{l,K}+h_{K}\,\|\bsigma\|^{2}_{l,K}+  \|u\|^{2}_{l+1,K}\Big)\\[2mm]\nonumber
& &+\sum_{K\in\cT^{d}_{h}}   h^{2l-1}_{K}
  \left( \dfrac{ \epsilon^{2} } {h_{K}} \,\|\nabla\cdot\bsigma\|^{2}_{l,K}
   +h_{K}\,\|\bsigma\|^{2}_{l,K}
   + \dfrac{ \epsilon } {h_{K}}  \,\|u\|^{2}_{l+1,K}\right).
\end{eqnarray*}
Now, (\ref{thm:5.3}) is a consequence of 
 (\ref{ineq:errtri}) and (\ref{ineqn:intererr:23w}). This completes the proof of the theorem. 
\end{proof}

Note that the a priori error estimate in Theorem~\ref{thm:apierr:w} is not optimal. 
This is because the coercivity of the homogeneous least-squares functionals in Lemma~\ref{lmm:disnorm:2w}
are established in a norm that is weaker than the norm used for the continuity of the functionals. 
To restore the full order of convergence, one may use
piecewise polynomials of degree $l+1$ to approximate $u$. 

\begin{thm}\label{thm:apierr:w2}
Let $(\bsigma^i_{h},u^i_{h})$,  $i=1,\,2,\,3$, be the solution of {\em (\ref{eqn:varprom})} with $\cu^h_i=(\Sigma^{l}_h\times V^{l+1}_h)
\cap\, \cu_i$. 
Under the assumption of {\em Theorem~\ref{thm:apierr:w}}, we have the following a priori error estimation:
 \begin{eqnarray} \nonumber
 && C_{i}\,\vertiii{(\bsigma-\bsigma^i_h,\,u-u^i_h)}^{2}_{i} \\ [2mm] \nonumber
 &\leq& \sum_{K\in\cT^{c}_{h}}  h^{2l}_{K}
    \Big(\|\nabla\cdot\bsigma\|^{2}_{l,K}+\|\bsigma\|^{2}_{l,K}+ h_K\, \|u\|^{2}_{l+2,K}\Big)\\[2mm] \label{thm:5.4}
 &&+\!\!\sum_{K\in\cT^{d}_{h}}   h^{2l}_{K}
  \left( \dfrac{ \epsilon^{2} } {h^2_{K}} \,\|\nabla\cdot\bsigma\|^{2}_{l,K}
   +\|\bsigma\|^{2}_{l,K}
   + \epsilon  \,\|u\|^{2}_{l+2,K}\right),
 \end{eqnarray}
where constants $C_{i}>0$ are independent of $\epsilon$.
\end{thm}
\begin{proof}
The a priori error estimate in (\ref{thm:5.4}) may be obtained in a similar fashion by noting that 
\[
\|u-u_I\|_{1} \leq C\,h^{l+1}\|u\|_{l+2}.
\]
\end{proof}

\section{Adaptive algorithm}
\label{sec:AMR}
\setcounter{equation}{0}

Asymptotic analysis (see, e.g., \cite{eck:79}) shows that the solution of a convection-dominated diffusion-reaction problem 
consists of two parts: the solution of the reduced equation ($\epsilon = 0$) and the correction, 
i.e., the boundary or interior layers. The boundary and interior layers are narrow regions where derivatives of the solution
change dramatically. For example, for the following problem~\cite{eck:79}:
  \begin{eqnarray*}
  \left\{
 \begin{array}{rl}
 -\epsilon\, \Delta u + \dfrac{\partial u}{\partial y} = f & \text{in}\,\, \Om  = (0,1)^2,\\[4mm]
 u = 0 & \text{on}\,\, \partial \Omega,
\end{array} \right.
\end{eqnarray*}
the exponential layer is of width $\cO(\epsilon)$ at $y = 1$, and the width of the parabolic boundary layers is $\cO(\epsilon^{1/2})$ 
at both $x = 0$ and $x = 1$.  Therefore, two sets of largely different scales exist simultaneously in the convection-dominated 
diffusion problem, and hence it is difficult computationally. 

On the one hand, one can apply the small scale over the entire domain, i.e., to use uniform fine meshes. 
With such a fine mesh, the standard Galerkin finite element method can also produce a good approximation. 
However, it is computationally inefficient due to the small region of the boundary and/or interior layers.
On the other hand, one can use the large scale over the entire domain. 
If the outflow boundary conditions are imposed strongly, the numerical solution (away from the boundary layers) will be polluted. 
In contrast, if the outflow boundary conditions are imposed weakly, the boundary layers can not be resolved 
(see, e.g., numerical results in~\cite{am09,chenfuliqiu}).

Neither of the above two approaches leads to a satisfactory numerical scheme. 
The failure is due to the fact that these approaches ignore this intrinsic property of the convection-dominated diffusion problem.
In contrast, the Shishkin mesh is aware of and respect it. 
Basically, the Shishkin mesh is a piecewise uniform mesh. 
In the diffusion-dominated region where the layers stand, it is a fine mesh suitable to the layer and in the convective region, 
it turns to be a coarse mesh. The disadvantage of the Shishkin mesh is that it needs the a priori information of the solution,
such as the location and the width of the layer, 
in order to construct a mesh of high quality. 
However, this information is not always available in advance, especially, for a complex problem. 

Based on the above considerations, we employ adaptive least-squares finite element methods. 
The least-squares estimators are simply defined as the value of the least-squares functionals at the current approximation.
To this end, for  each element $K\in\cT_h$, denote the local least-squares functionals by
\begin{eqnarray*}
G^h_{1,K}(\btau,\,v;\,f) 
 &=&  \| \btau+\epsilon^{1/2}\,\nabla v\|^{2}_K+\|\epsilon^{1/2}\,\nabla\cdot\btau+\bbeta\cdot\nabla v+c\, v-f\|^{2}_K, \\[4mm]
G^h_{2,K}(\btau,\,v;\,f) 
&=& \left\{\begin{array}{ll}
G^h_{1,K}(\btau,\,v;\,f),&\quad\text{if}\,K\cap \Gamma_{+} = \emptyset,  \\[3mm]
G^h_{1,K}(\btau,\,v;\,f)+\sum\limits_{e\in K\cap\Gamma_{+}} h^{-1}_e \|\epsilon^{-1/2} v\|^2_{0,\,e},&\quad\text{otherwise},
\end{array}
\right. \\[4mm]
\mbox{and }\,\, G^h_{3,K}(\btau,\,v;\,f) 
 &=& \left\{\begin{array}{ll}
G^h_{1,K}(\btau,\,v;\,f),&\quad\text{if}\,K\cap \Gamma_{+} = \emptyset,  \\[3mm]
G^h_{1,K}(\btau,\,v;\,f)+\sum\limits_{e\in K\cap\Gamma_{+}} h^{-1}_e \|v\|^2_{0,\,e},&\quad\text{otherwise}.
\end{array}
\right.
\end{eqnarray*}
Let $(\hat{\bsigma}^h_i,\,\hat{u}^h_i)$ be the current approximations to the solutions of (\ref{eqn:varprom}) for $i=1,\,2,\,3$. 
Then the least-squares indicators are simply the square root of the value of the local least-squares functionals at the current approximation:
 \beq\label{indicator}
  \eta^{i}_K  = G^h_{i, K}\,(\hat{\bsigma}^{i}_{h},\,\hat{u}^{i}_{h};\,f)^{1/2}
  \eeq
for all $K\in \cT_h$ and for $i=1,\,2,\,3$. The least-squares estimators are
 \beq\label{estimator}
 \eta^{i} 
 = \left( \sum_{K\in \cT_h}  \big(\eta^{i}_K \big)^2\right)^{1/2}
 = G^h_{i}\,(\hat{\bsigma}^{i}_{h},\, \hat{u}^{i}_{h};\,f)^{1/2}
 \eeq
for $i=1,\,2,\,3$. 

Let $(\bsigma,\,u)$ be the solution of (\ref{eqn:vf}) and denote the true errors by 
\[
\hat{\bE}^i= \bsigma-\hat{\bsigma}^i_{h}
\quad  \mbox{and} \quad \hat{e}^i= u-\hat{u}^1_{h}
\quad\mbox{for}\quad i=1,\,2,\,3.
\]

\begin{thm}\label{thm:postestor1}
There exist positive constants $C_{e,1}$ and $C_{r,1}$ independent of $\epsilon$ such that 
 \begin{eqnarray}\label{ineq:coeff1}
 \eta^{1}_K \leq 
 C_{e,1}\, \left(M_{1,K}(\hat{\bE}^1,\, \hat{e}^{1}) 
 + \| \bbeta\cdot\nabla \, \hat{e}^1\|_{K}^2 + \epsilon\,\|\nabla\cdot \hat{\bE}^1\|^2_K\right)^{1/2}
\end{eqnarray}
for all $K\in\cT$ and that 
\begin{equation}\label{ineq:relia1}
 M_{1}(\hat{\bE}^1,\, \hat{e}^{1})^{1/2}
 \leq C_{r,1}\, \eta^{1}.
\end{equation}
\end{thm}
\begin{proof}
Since the exact solution $(\bsigma,\,u)$ satisfies (\ref{sys:dif:eq1}), we have
\[ 
 \big(\eta^{1}_K\big)^2 
 = G^h_{1,K}(\hat{\bE}^1,\, \hat{e}^{1};\,0)
 \quad\mbox{and}\quad
 \big(\eta^{1}\big)^2 =
 G^h_{1}(\hat{\bE}^1,\, \hat{e}^{1};\,0).
\] 
which, together with the triangle inequality and Theorem~\ref{thm:2w},
imply the efficiency and the reliability bounds, respectively.
\end{proof}
\begin{thm}\label{thm:postestor}
There exist positive constants $C_{e,i}$ independent of $\epsilon$ such that 
\begin{eqnarray}\label{ineq:coeff}
C_{e,\,i}\,\big(\eta^{i}_K \big)^2
  \leq M^h_{i,K}(\hat{\bE}^i,\,  \hat{e}^{i}) 
 + \| \bbeta\cdot\nabla \hat{e}^{i}\|_{K}^2 + \epsilon\,\|\nabla\cdot \hat{\bE}^i\|^2
\end{eqnarray}
for all $K\in\cT$ and $i=2,\,3$.
\end{thm}
\begin{proof}
Let $\alpha_i= -1$ for $i=2$ or $0$ for $i=3$. With the fact that $(\bsigma,\,u)$ is the exact solution satisfying (\ref{sys:dif:eq1}), we have
\begin{eqnarray}\nonumber
& &\eta^{i}(\hat{\bsigma}^h_i,\,\hat{u}^h_i)^2 = G^h_{i}(\hat{\bsigma}^h_i,\,\hat{u}^h_i;\,f)\\[2mm] \nonumber
&=& \| \hat{\bsigma}^h_i+\epsilon^{1/2}\,\nabla \hat{u}^h_i\|^{2}+\|\epsilon^{1/2}\,\nabla\cdot\hat{\bsigma}^h_i+\bbeta\cdot\nabla \hat{u}^h_i+c\, \hat{u}^h_i-f\|^{2}+\sum_{e\in\cE_h\cap\Gamma_{{+}}}\epsilon^{\alpha_i}\,h^{-1}_{e}\|\hat{u}^h_i\|^{2}_{0,e}\\[2mm] \nonumber
&=& \| \hat{\bE}^i+\epsilon^{1/2}\,\nabla \hat{e}^i\|^{2}+\|\epsilon^{1/2}\,\nabla\cdot\hat{\bE}^i+\bbeta\cdot\nabla \hat{e}^i+c\, \hat{e}^i\|^{2}+\sum_{e\in\cE_h\cap\Gamma_{{+}}}\epsilon^{\alpha_i}\,h^{-1}_{e}\|\hat{e}^i\|\\[2mm] \label{eqn:posterr}
&=& G^h_{i}(\hat{\bE}^i,\,\hat{e}^i;\,0),
\end{eqnarray}
with which, 
the efficiency bound simply follows from (\ref{eqn:posterr}) and the Cauchy-Schwarz inequality.
\end{proof}

In the remainder of this section, we describe the standard adaptive mesh refinement algorithm.
Starting with an initial triangulation $\cT_0$, a sequence of nested triangulations $\{\cT_{l}\}$ is generated through the well known AFEM-Loop:
\[
\text{\textbf{SOLVE}}  \longrightarrow \text{\textbf{ESTIMATE}} \longrightarrow \text{\textbf{MARK}} \longrightarrow \text{\textbf{REFINE}}.
\]

The \textbf{SOLVE} step solves (\ref{eqn:varprom}) in the finite element space corresponding to the mesh $\cT_l$ for a numerical approximation $(\bsigma^i_h(l),\,u^i_h(l)) \in \cu^h_i(l)$, where $\cu^h_i(l)$ is the finite element space defined on $\cT_l$. Hereafter, we shall explicitly express the dependence of a quantity on the level $l$ by either the subscript like $\cT_l$ or the variable like $\cu^h_i(l)$.

The \textbf{ESTIMATE} step computes the indicators $ \{\eta^{i}_K (l)\}$ and the estimator $\eta^{i} (l)$ defined in (\ref{indicator}) and (\ref{estimator}), respectively. 

The way to choose elements for refinement influences the efficiency of the adaptive algorithm. If most of elements are marked for refinement, 
then it is comparable to uniform refinement, which does not take full advantage of the adaptive algorithm and results in redundant degrees of freedom. On the other hand, if few elements are refined, then it requires many iterations, which undermines the efficiency of the adaptive algorithm, since each iteration is costly. For the singularly perturbed problems, it is well known that the indicators associated with the elements in the layer region are much larger than others. Therefore, we \textbf{MARK} by the maximum algorithm, which defines the set $\hat{\cT}_l$ of marked elements such that for all $K \in \hat{\cT}_l$
\[
\eta^{i}_K (l)  \ge \theta\,\max_{K\in\cT_l} \eta^{i}_K (l).
\]

The \textbf{REFINE} step is to bisect all the triangles in $\hat{\cT}_l$ into two sub-triangles to generate a new triangulation $\cT_{l+1}$. Note that some triangles in $\cT_l\, \backslash\, \hat{\cT}_l$ adjacent to triangles in $\hat{\cT}_l$ are also refined in order to avoid hanging nodes.

In summary, the adaptive least-squares finite element algorithm can be cast as follows: with the initial mesh $\cT_0$, marking parameter $\theta\in(0,1)$, and the maximal number of iteration $maxIt$, for $l = 0, 1, \cdots, maxIt$, do
\begin{enumerate}
\item[(1)] $(\bsigma^i_h(l),\,u^i_h(l)) = \text{\textbf{SOLVE}}(\cT_l);$
\item[(2)] $\{\eta^{i}_K (l)\} = \text{\textbf{ESTIMATE}}(\cT_l,\,\bsigma^i_h(l),\,u^i_h(l));$
\item[(3)] $\hat{\cT_l} = \text{\textbf{MARK}}(\cT_l,\,\{\eta^{i}_K (l)\});$
\item[(4)] $\cT_{l+1} = \text{\textbf{REFINE}}(\cT_l,\, \hat{\cT_l}).$
\end{enumerate}

\section{Numerical experiments}
\label{sec:numtest}
\setcounter{equation}{0}

In this section, we conduct several numerical experiments on two model problems used by many authors (see, e.g.,~\cite{am09,chenfuliqiu}). 
Both the model problems are defined in the unit square and all numerical experiments are started with the same initial mesh, which consists of sixteen isosceles right triangles. The marking parameter $\theta$ is chosen to be $0.6$.

\subsection{Boundary layer}
\label{exam:blayer}

In this example, $\bbeta  = [1,1]^{T}$, and $c=0$, and the external force $f$ is chosen such that the exact solution is
\[
u(x,y) = \sin\frac{\pi x}{2}+\sin\frac{\pi y}{2}\Big(1-\sin\frac{\pi x}{2}\Big)+\frac{e^{-1/\epsilon}-e^{-(1-x)(1-y)/\epsilon}}{1-e^{-1/\epsilon}}.
\]
This solution is smooth, but develops boundary layers at $x = 1$ and $y = 1$ with width $\cO(\epsilon)$. This example is suitable for testing
capability of the numerical approximations on resolving the boundary layers.

In this numerical experiment, $\epsilon = 10^{-3}$.  Given the tolerance $tol = 0.5$, computation is terminated if
\begin{eqnarray}\label{ine:stop}
\eta^i(l) \leq tol.
\end{eqnarray}
Since the exact solution is available, the true error is computed and the effectivity index is defined as follows: 
\begin{eqnarray}
\text{eff-index} := \frac{\eta^i(\bsigma^i_h,\,u^i_h)}{\vertiii{(\bsigma-\bsigma^{i}_{h},\, u - u^{i}_{h})}_{i}}.
\end{eqnarray}
\begin{figure}[h!]

\begin{subfigure}{0.5\textwidth}
\centering
\includegraphics[width = \textwidth]{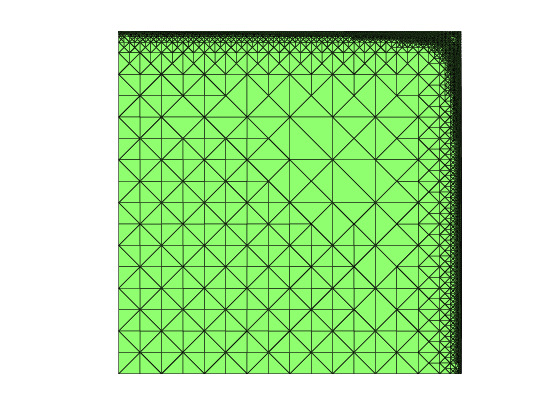}
\end{subfigure}
\begin{subfigure}{0.5\textwidth}
\centering
\includegraphics[width = \textwidth]{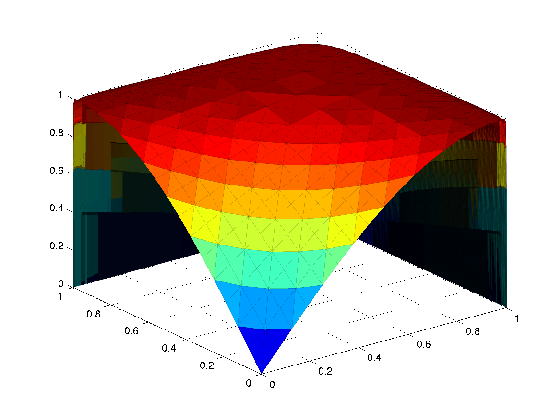}
\end{subfigure}
\begin{subfigure}{0.5\textwidth}
\centering
\includegraphics[width = \textwidth]{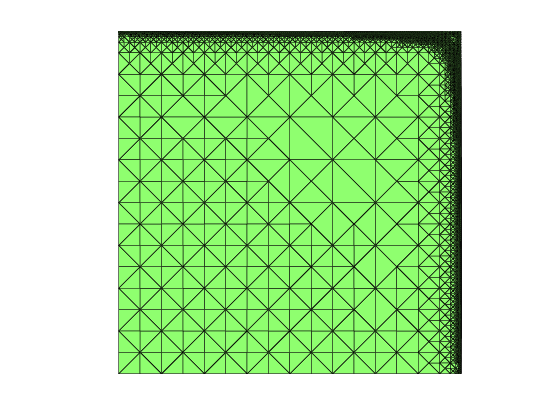}
\end{subfigure}
\begin{subfigure}{0.5\textwidth}
\centering
\includegraphics[width = \textwidth]{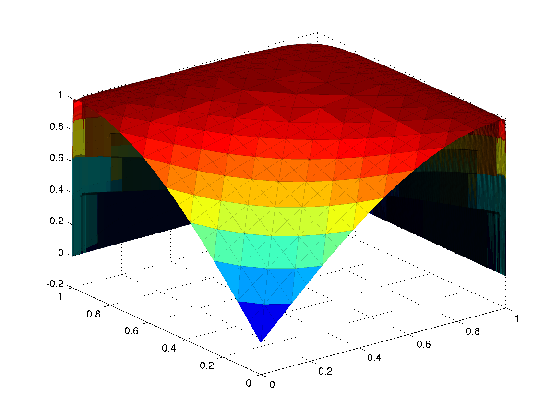}
\end{subfigure}
\begin{subfigure}{0.5\textwidth}
\centering
\includegraphics[width = \textwidth]{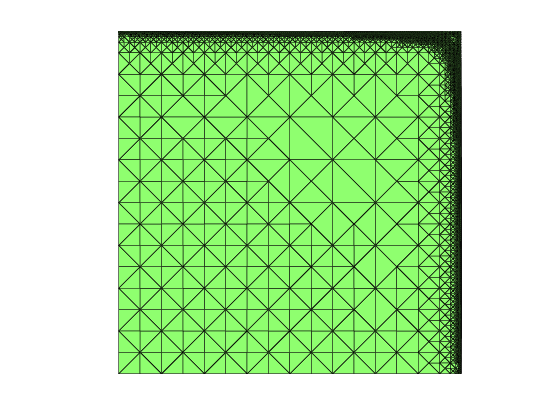}
\end{subfigure}
\begin{subfigure}{0.5\textwidth}
\centering
\includegraphics[width = \textwidth]{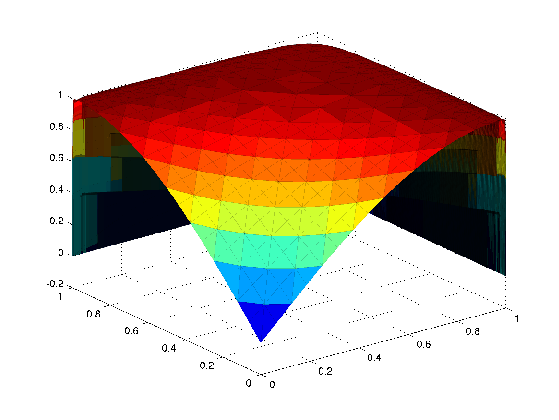}
\end{subfigure}
\caption{The final meshes and the numerical solutions are, respectively, displayed in the first and the second columns
and the rows are corresponding to $i=1,\,2,\,3$.}
\label{fig:ex1:solu}
\end{figure}
The final meshes are displayed in the first column of Figure~\ref{fig:ex1:solu} 
when the stopping criterion (\ref{ine:stop}) is satisfied. They clearly show that the refinements cluster around the boundary layer area. 
The numerical solutions on the final meshes are depicted in the second column of Figure~\ref{fig:ex1:solu}. 
All the three methods successfully capture the sharp boundary layers, and no visible oscillation appears in the numerical solutions. 
Reported in Figure~\ref{fig:ex1:conv} is the convergence rates of the numerical solutions. The errors with the norm $\vertiii{\cdot}_i$ that are used in the a priori error estimate are tracked, which converge in the order of $(DoF)^{-1}$. Moreover, the convergence rate is independent of the value of $\epsilon$. This is also verified by the test problem with $\epsilon = 10^{-4}$, where the convergence rate does not deteriorate 
(see the second column of Figure~\ref{fig:ex1:conv}). 
\begin{figure}[h!]

\begin{subfigure}{0.5\textwidth}
\centering
\includegraphics[width = \textwidth]{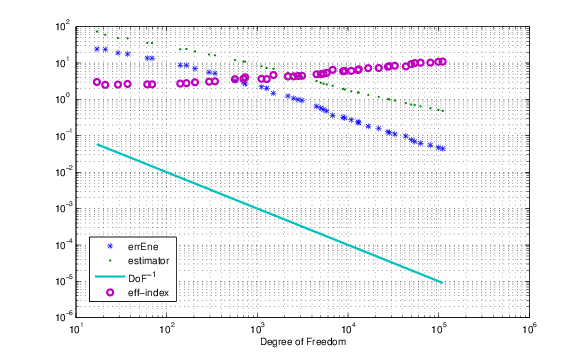}
\end{subfigure}
\begin{subfigure}{0.5\textwidth}
\centering
\includegraphics[width = \textwidth]{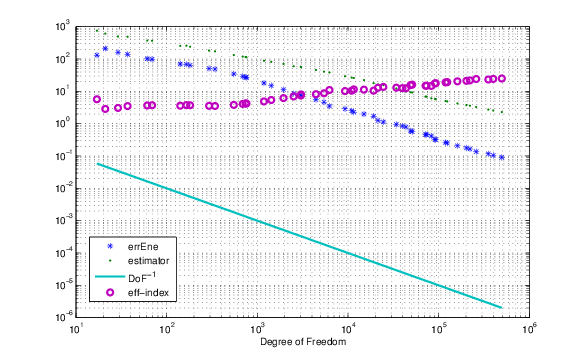}
\end{subfigure}
\begin{subfigure}{0.5\textwidth}
\centering
\includegraphics[width = \textwidth]{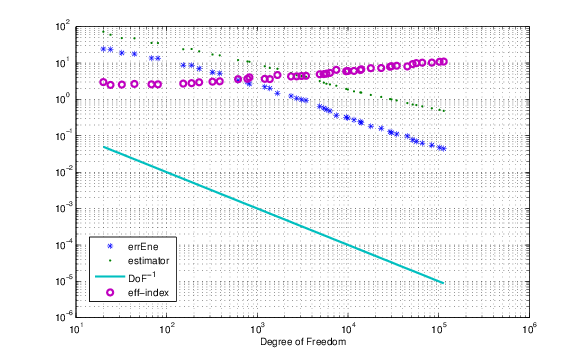}
\end{subfigure}
\begin{subfigure}{0.5\textwidth}
\centering
\includegraphics[width = \textwidth]{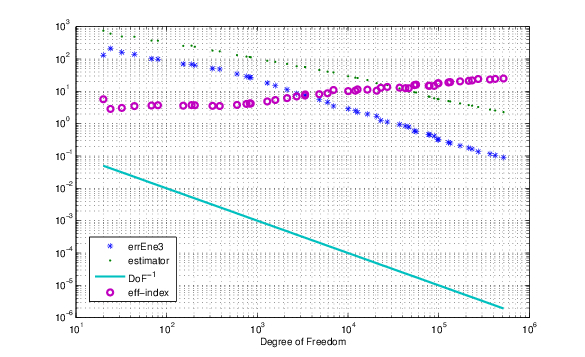}
\end{subfigure}
\begin{subfigure}{0.5\textwidth}
\centering
\includegraphics[width = \textwidth]{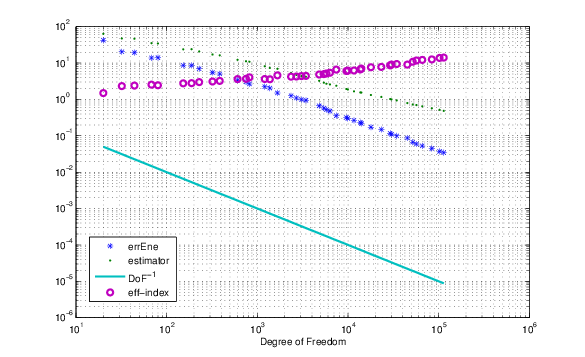}
\end{subfigure}
\begin{subfigure}{0.5\textwidth}
\centering
\includegraphics[width = \textwidth]{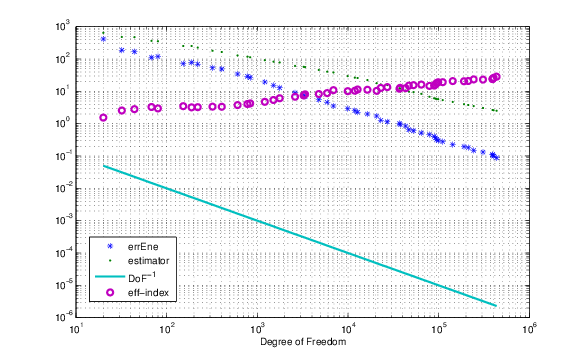}
\end{subfigure}
\caption{The convergence rates corresponding to $\epsilon = 10^{-3}$ and $10^{-4}$
are displayed in the first and the second columns, respectively,
and the rows are corresponding to $i=1,\,2,\,3$.}
\label{fig:ex1:conv}
\end{figure}

\subsection{Interior layer}
\label{exam:intlayer}
In the second example, $\bbeta = [1/2, \sqrt{3}/2]^{T}$, $c=0$, $f = 0$, and the boundary condition is
\[
u = \left\{
\begin{array}{cl}
1,&\mbox{on}\, \{(x,\,y) : y = 0, 0\leq x\leq 1\},\\[2mm]
1,&\mbox{on}\, \{(x,\,y) :  x = 0, 0\leq y\leq 1/5\},\\[2mm]
0,&\mbox{otherwise}.
\end{array}
\right.
\]
The exact solution of the problem remains unknown. However, it is known that, additional to the boundary layers, the solution develops an interior layer along the line $ y = \sqrt{3}\, x + 0.2$ due to the discontinuity at $(0,0.2)$ of the boundary condition. The problem is chosen to test whether the formulations can capture the interior layers.

Figure~\ref{fig:ex2:solu} shows that all the three methods capture both the boundary and the interior layers. 
Moreover, the numerical solutions do not exhibit any visible oscillation, which is much better than the results reported in ~\cite{am09}.
\begin{figure}[h!]

\begin{subfigure}{0.32\textwidth}
\centering
\includegraphics[width = \textwidth]{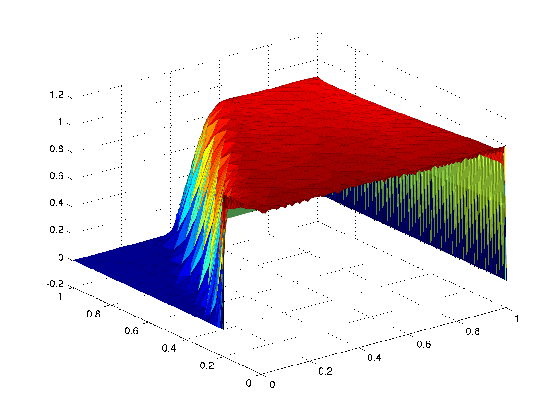}
\end{subfigure}
\begin{subfigure}{0.32\textwidth}
\centering
\includegraphics[width = \textwidth]{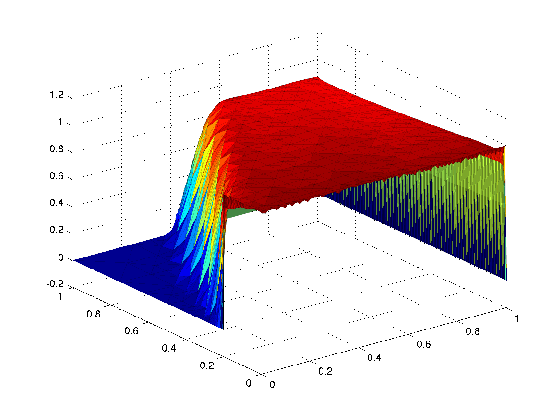}
\end{subfigure}
\begin{subfigure}{0.32\textwidth}
\centering
\includegraphics[width = \textwidth]{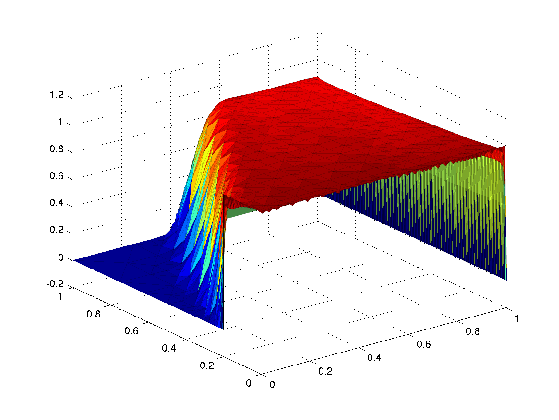}
\end{subfigure}
\caption{Numerical solutions corresponding to $i = 1,\,2,\,3$ from left to right.}
\label{fig:ex2:solu}
\end{figure}

\section*{Acknowledgements}

We thank Dr. Shuhao Cao for the discussion and helpful suggestions on the computation of the test problems.

\end{document}